\documentclass[11pt]{amsart}

\usepackage{amsmath, amsfonts, amsthm, amssymb, mathtools}
\usepackage{verbatim}
\usepackage[margin=1.25in]{geometry}
\usepackage[all,arc]{xy}
\usepackage{enumerate}
\usepackage{mathrsfs}
\usepackage{cleveref}
\usepackage{cite}
\DeclareMathOperator{\SL}{SL}
\newcommand{\mscr}[1]{\ensuremath{\mathscr{#1}}}
\newcommand{\mc}[1]{\ensuremath{\mathcal{#1}}}

\newcommand{\Z}{\ensuremath{\mathbb{Z}}}

\newcommand{\R}{\ensuremath{\mathbb{R}}}
\newcommand{\C}{\ensuremath{\mathbb{C}}}

\renewcommand{\H}{\ensuremath{\mathcal{H}}}

\DeclareMathOperator{\Tr}{Tr}

\newtheorem{thm}{Theorem}[section]
\newtheorem{cor}[thm]{Corollary}

\newtheorem{lem}[thm]{Lemma}

\theoremstyle{definition}

\newtheorem*{exmp}{Example}

\newtheorem*{question}{Question}

\theoremstyle{remark}
\newtheorem*{rem}{Remark}

\numberwithin{equation}{section}

\title{Algebraic Relations between Partition Functions and the $j$-Function}
\author{Alice Lin, Eleanor McSpirit, and Adit Vishnu}

\begin{document}

\begin{abstract}
We obtain identities and relationships between the modular $j$-function, the generating functions for the classical partition function and the Andrews $spt$-function, and two functions related to unimodal sequences and a new partition statistic we call the ``signed triangular weight" of a partition. These results follow from the closed formula we obtain for the Hecke action on a distinguished harmonic Maass form $\mathscr{M}(\tau)$ defined by Bringmann in her work on the Andrews $spt$-function. This formula involves a sequence of polynomials in $j(\tau)$, through which we ultimately arrive at expressions for the coefficients of the $j$-function purely in terms of these combinatorial quantities.
%which ties partition statistics to the modular $j$-function.
%spt, unimodal, start with ptn, surprisingly, determine hecke action on related hmf whcih ties all of these together, allow us to find formulas for coeffs of j
%we explore where these functions appear in the theory of modular forms in order to ...
%We obtain congruences modulo $\ell$ relating unimodal sequences and the signed triangular weight, following congruences for $spt$ proved by Ono \cite{ono2011spt}.
\end{abstract}

\maketitle
\section{Introduction and Statement of Results}

Partitions, first and foremost combinatorial objects, permeate seemingly disparate areas of mathematics. 
%The representation theory of the symmetric group $S_n$ is one of the first deep examples in which the combinatorial and number-theoretic properties of partitions enlightened us to structures present in representation theory. The representation theory of the symmetric group depends centrally on the Young tableaux corresponding to partitions. 
The partition function $p(n)$ gives the number of ways to write $n$ as the sum of unordered positive integers. The generating function for $p(n)$ is a weakly holomorphic modular form of weight $-1/2$, namely
\begin{align}
\label{eqn: partition generator}
    \mscr{P}(q) := \sum_{n \ge 0} p(n)q^{24n-1} = q^{-1} \prod_{n \geq 1} \frac{1}{1-q^{24n}} = \frac{1}{\eta(24\tau)},
\end{align}
where $\eta(\tau)$ is Dedekind's eta-function and we use the convention $q = e^{2\pi i\tau}$. This is one indication of partitions' deep ties to number theory. 
%It also comes as no surprise that the partition function $p(n)$, whose generating function is a modular form of weight $1/2$, is a fundamental object occurring in number theory. 
Outside combinatorics and number theory, perhaps the most prominent role for partitions is in representation theory, where the theory of Young tableaux for partitions encodes the irreducible representations of all symmetric groups \cite[Theorem~2.1.11]{james1981representation}. 

Other modular forms and functions that were first studied in number theory have likewise appeared in the representation theory of finite groups. In particular, the modular $j$-function, whose Fourier expansion is
\begin{align}
    j(\tau) = \sum_{n \ge -1} c(n)q^n = q^{-1} + 744 + 196884q + 21493760 q^2 + \cdots, \label{eqn:j fn}
\end{align}
is well-known in number theory because the $j$-invariants, i.e. the values of $j(\tau)$ for $\tau \in \H$, parametrize isomorphism classes of elliptic curves over $\C$ \cite[Proposition~12.11]{silverman2009arithmetic}. %namely the elliptic curve defined by the torus $\C/(\Z \oplus \Z\tau)$. 

McKay famously observed that the first few coefficients of $j(\tau)$ satisfy striking relations such as %and Thompson famously saw that 
\begin{equation}
\label{eqn: mckay c(1) and c(2)}
\begin{aligned}
    c(1)&=196884=1+196883, \\
    c(2)&=21493760=1+196883+21296876,
\end{aligned} 
\end{equation}
where the right-hand sides are linear combinations of dimensions of irreducible representations of the monster group $M$. 
Such expressions inspired Thompson to conjecture \cite{thompson1979num} that there is a monstrous moonshine module, an infinite-dimensional graded $M$-module $V^\natural = \bigoplus_{n \gg -1} V_n$ 
%whose automorphisms are the monster group, 
such that for $n \ge -1$, we have
\begin{align*}
    c(n) = \dim(V_n).
\end{align*}
Thompson further conjectured that, since the graded dimension is the graded trace of the identity element of $M$, the traces of other elements $g$ may likewise be related to naturally-occuring $q$-series. This was refined by Conway and Norton in \cite{conway1979monstrous}, who conjectured that for every element $g \in M$, the McKay-Thompson series 
\[
T_g(\tau) := \sum_{n=-1}^{\infty}\Tr(g|V_n)q^n
\]
%This observation motivated an exploration of the connection of the $j$-function to the representation theory of the monster module $V^\natural$. 
is the Hauptmodul which generates the function field for a genus 0 modular curve for a particular congruence subgroup $\Gamma_g \subset \SL_2(\mathbb{R})$.
%% TODO: GET THIS CHECKED PLS^
Borcherds proved the Conway--Norton conjecture for the Monster Moonshine Module in \cite{borcherds1992monstrous}, an impactful result which, in part, solidifies the $j$-function's connection to the representation theory of $M$. 
%encodes the graded dimensions of the graded representations of $M$. 

%% TODO: ASK ABOUT THE ABOVE SENTENCE!!!
Since the $j$-function and partitions appear in both number theory and representation theory, one can ask if there is a relation between $c(n)$ and $p(n)$. In this paper, we discover that the coefficients of the Fourier expansion of both the $j$-function and a certain sequence of polynomials in $j$ have a combinatorial description in terms of partitions of integers and unimodal sequences. This suggests the possibility of deeper connections between the representation theory of the symmetric group and the monster Lie algebra.

%This research began by exploring the theory of harmonic Maass forms, which are...
This research is inspired by recent work of Andrews \cite{andrews2008number} in which he defined $spt(n)$ to count the number of smallest parts among all integer partitions of $n$. For example, we can determine that $spt(4) = 10$ by counting the following underlined parts across all five partitions of 4:
\[
    \underline{4} = 3+\underline{1}
    = \underline{2} + \underline{2}
    = 2 + \underline{1} + \underline{1}
    = \underline{1} + \underline{1} + \underline{1} + \underline{1}.
\]
Following the notation of \cite{ono2011spt}, we define a renormalized generating function for $spt(n)$ as 
\begin{align}
\label{eqn: spt generator}
    \mathcal{S}(q):=\sum_{n \geq 1} spt(n)q^{24n-1}.
\end{align}

Paralleling Ramanujan's notable congruences 
\begin{align*}
    p(5n + 4) &\equiv 0 \pmod{5}, \\
    p(7n + 5) &\equiv 0 \pmod{7}, \\
    p(11n + 6) &\equiv 0 \pmod{11}, 
\end{align*}
Andrews \cite{andrews2008number} showed that the $spt$ function satisfies the congruences
\begin{align*}
    spt(5n + 4) &\equiv 0 \pmod{5}, \\
    spt(7n + 5) &\equiv 0 \pmod{7}, \\
    spt(13n + 6) &\equiv 0 \pmod{13}.
\end{align*}
Of further interest are the $spt$-function's rich families of congruences modulo all primes $\ell \ge 5$.
%Let $\left( \frac{\cdot}{\ell} \right) $ denote the Legendre symbol.
As Ono proved in \cite{ono2011spt},
if $\ell \ge 5$ is prime, $n \ge 1$, and $\left(\frac{-n}{\ell} \right) = 1$, then
\begin{align}
    spt \left( \frac{\ell^2 n - 1}{24} \right) \equiv 0 \pmod{\ell}. \label{eqn:ono's spt cong}
\end{align}
Subsequent work by Ahlgren et al. in \cite{ahlgren2011ell} extended these congruences to arbitrary powers of $\ell$. If $m \ge 1$, then
\begin{align}
    spt\left( \frac{\ell^{2m}n + 1}{24}\right) \equiv 0 \pmod{\ell^m}.
    \label{eqn: spt cong for powers}
\end{align}

These congruences follow from studying a distinguished harmonic Maass form $\mscr{M}(\tau)$ defined by Bringmann in \cite{bringmann2008explicit} (see (\ref{eqn: script M whole thing})). For background on harmonic Maass forms, we refer the reader to \cite{bringmann2017harmonic} and \cite{ono2009unearthing}.
The function $\mscr{M}(\tau)$ is of particular interest because its holomorphic part $M^+(\tau)$ involves the generating functions for both $p(n)$ and $spt(n)$; namely we have
\begin{align}
    M^+(\tau) = \mc{S}(q) + \frac{1}{12}q \frac{d}{dq} \mscr{P}(q).
\end{align}
%% TODO: NEED TO ASK ABOUT WHICH ONES TO CITE FOR CONGRUENCES MOD ell

For weight $3/2$ harmonic Maass forms with Nebentypus $\chi_{12}:=\left(\frac{12}{\cdot} \right)$, we follow the normalization given in \cite{ono2011spt} to define the Hecke operators $T(\ell^2)$ of index $\ell^2$ on a power series  $f(\tau)~=~\sum_{n \gg -\infty} a(n)q^n$ by
\begin{align}
     f(\tau) \mid T(\ell^2) := \sum_{n \gg -\infty} \left[ a(\ell^2n) + \left( \frac{3}{\ell} \right)\left( \frac{-n}{\ell} \right)a(n) + \ell a(n/\ell^2) \right] q^n.
\end{align}
The congruences in (\ref{eqn:ono's spt cong}) and (\ref{eqn: spt cong for powers}) follow from the fact that
\begin{align}
\label{eqn: ono M cong}
    M^+(\tau)\mid T(\ell^2) \equiv \left( \frac{3}{\ell} \right) M^+(\tau) \pmod{\ell}.
\end{align}

Ono asked whether there exist explicit identities which imply (\ref{eqn: ono M cong}). We answer this question. Using the standard notation $(q;q)_\infty := \prod_{n \ge 1}(1-q^n)$, we define a sequence of monic integer polynomials $B_m(x)$ of degree $(m-1)$ by
\begin{equation} \label{eqn: B_m defined}
\begin{aligned}
    \mc{B}(x,q) = \sum_{m \ge 1}B_m(x)q^m &:= (q;q)_\infty \cdot\frac{1}{j(\tau) - x} \\
    &= q + (x -745) q^2 + (x^2 - 1489 x + 357395) q^3 + \cdots.
\end{aligned}
\end{equation}
%(do not define J in the intro)
%We also write $\mscr{P}(q) = \sum_{n \ge 0}p(n) q^{24n - 1} = 1/\eta(24\tau)$ as the generating function for the partition function, where $\eta(\tau) = q^{1/24}(q;q)_\infty$ is Dedekind's eta-function. 
In terms of the Eisenstein series $E_4(\tau)$ and $E_6(\tau)$, as well as Ramanujan's Delta function $\Delta(\tau)$, we offer the following solution to Ono's problem. 
\begin{thm}
\label{thm: 1.1 Hecke}
    If $\ell \geq 5$ is a prime and $\delta_\ell := \frac{\ell^2 - 1}{24}$, then
    \begin{align*}
        M^+(\tau) \mid_{3/2} T(\ell^2) = \left(\frac{3}{\ell} \right)(1+\ell)M^+(\tau) - \frac{\ell}{12} \mscr{P}(q) \cdot B_{\delta_\ell}(j(24\tau))\cdot \frac{E_4^2(24\tau)E_6(24\tau)}{\Delta(24\tau)}.
    \end{align*}
\end{thm}

\begin{rem}
We note that the identity in the theorem immediately reduces to (\ref{eqn: ono M cong}) modulo $\ell$. Moreover, this result gives an expression for the Hecke action in terms of only the original mock modular form and the coefficient of $q^{-\ell^2}$ produced by the Hecke operator. Therefore, the resulting mock modular form is determined by a single term. 
% Therefore, this operator which is d
\end{rem}

For notational clarity, we note that 
    \begin{align*} 
    \label{eqn: ddtau cn}
         -q \frac{d}{dq}j(\tau) = \frac{E_4^2(\tau)E_6(\tau)}{\Delta(\tau)} &= q^{-1} - \sum_{n \ge 1} nc(n)q^n
        = q^{-1} - 196884 q - 42987520 q^2 + \cdots .
    \end{align*}
    Thus, $B_{\delta_\ell}(j(24\tau))\cdot \frac{E_4^2(24\tau)E_6(24\tau)}{\Delta(24\tau)}$ is completely determined by the coefficients of $j$. For convenience, we write
    \begin{align*}
        %\frac{E_4^2(24\tau)E_6(24\tau)}{\Delta(24\tau)} &=
        -q \frac{d}{dq}j(24\tau)=q^{-24}-196884q^{24}-42987520q^{48}+\cdots. 
    \end{align*}

%Although it is not immediately clear that the coefficients of the $j$-function appear outside of the term $B_{\delta_\ell}(j(24\tau))$, we in fact have that 

\begin{exmp}
    %To demonstrate the theorem above, if we define $M_\ell(\tau):= M(\tau) \mid_{3/2} T(\ell^2) - \left(\frac{3}{\ell} \right)(1+\ell)M(\tau)$, then
    Here we illustrate Theorem \ref{thm: 1.1 Hecke} for the primes $5$, $7$, and $11$.
    In the notation of \cite{ono2011spt}, we define
    \begin{align}
        M_\ell(\tau) := M^+(\tau)\mid_{3/2}T(\ell^2)-\left( \frac{3}{\ell}\right)(1+\ell)M^+(\tau).
    \end{align}
    For $\ell = 5$, note that $\delta_5 = 1$ and $B_1(x) = 1$. Therefore, we find that
    \begin{align*}
        M_5(\tau) 
        =\frac{5}{12} \mscr{P}(q) \cdot q\frac{d}{dq}j(24\tau)= %M(\tau) \mid T(25) - \left(\frac{3}{5} \right)(1+5)M(\tau) &=
        - \frac{5}{12}q^{-25} - \frac{5}{12}q^{-1} + \frac{492205}{6}q^{23} + \cdots.
        %+ \frac{215922005}{12}q^{47} + \cdots.
    \end{align*}
    For $\ell = 7$, $\delta_7 = 2$ and $B_2(x) = x - 745$. Therefore, we have
    \begin{align*}
        %M_7(\tau) &= %M(\tau) \mid T(49) - \left(\frac{3}{7} \right)(1+7)M(\tau) &=
        %- \frac{7}{12}q^{-49} - \frac{7}{12}q^{-1} + \frac{149078125}{12}q^{23} + \cdots, \\
        M_7(\tau)= \frac{7}{12} \mscr{P}(q) \cdot (j(24\tau) - 745) \cdot q\frac{d}{dq}j(24\tau) = - \frac{7}{12}q^{-49} - \frac{7}{12}q^{-1} + \frac{149078125}{12}q^{23} + \cdots.
    \end{align*}
    For $\ell = 11$, $\delta_{11} = 5$ and 
    \begin{align*}
    B_5(x) = x^4 - 2977x^3 + 2732795 x^2 - 812685832 x + 4947668669.
    \end{align*}
Therefore, we have
    \begin{align*}
       % M_{11}(\tau) &= - \frac{11}{12} q^{-121} + \frac{11}{12} q^{-1} + \cdots \\
        M_{11}(\tau)= \frac{11}{12}\mscr{P}(q) \cdot B_5(j(24\tau)) \cdot q\frac{d}{dq}j(24\tau) 
        =- \frac{11}{12} q^{-121} + \frac{11}{12} q^{-1} + \cdots.
    \end{align*}
\end{exmp}

In view of (\ref{eqn: ddtau cn}), the case $\ell = 5$ gives an expression for $M^+(\tau) \mid T(25)$ in terms of the coefficients $c(n)$ of the $j$-function, thus deriving an unexpected relationship between these coefficients and the values of $p(n)$ and $spt(n)$. %Using this knowledge, we can express
Namely, we offer the following partition-theoretic counterparts to (\ref{eqn: mckay c(1) and c(2)}):
\begin{align*}
    c(1) = 196884
    &=2+49+15708+181125,
    %\\
    %&=2+\frac{84spt(1)+161p(1)}{5}+\frac{12}{5}spt(24)+115p(24)
    \\
    c(2) = 21493760
    &=\frac{1}{2}\Big(1-49+182-15708-181125+2405844+40778375\Big).
    %\\
    %&=\frac{1}{2}\Big( 1-\frac{84spt(1)+161p(1)}{5}+\frac{7spt(2)+329p(2)}{5}\\
    %&\hspace{1.2cm}-\frac{12}{5}spt(24)-115p(24)+\frac{12}{5}spt(49)+235p(49)\Big)
\end{align*}

The two identities above are examples of a more general theorem. To make this precise, it is important to illustrate how the summands above correspond to $p(n)$ and $spt(n)$. We require the following notation. For $n \ge 1$, we define
\begin{equation}
\label{h coefficients}
\begin{aligned}
    h_1(24n-1) &:=\frac{12}{5}spt(25n-1)+5(24n-1)p(25n-1)\\
    &\hspace{.5cm} +\mu_n \cdot  \Big(\frac{12}{5} spt(n) + 5(24n-1)p(n)\Big),\\
    h_2(25(24n-1)) &:= 12spt(n)+(24n-1) p(n),
\end{aligned}
\end{equation}
where $\mu_n:= 6-\left(\frac{1-24n}{5}\right)$.
We define $h_1(m)=0$ if $m \not\equiv 23 \bmod{24}$ and $h_2(m)=0$ if $m \not\equiv 23 \bmod{24}$ or if $m \not\equiv 0 \bmod{25}$. We will also need the following function. For $n=1$, we set $s(n)=2$, and for $n>1$, let
\begin{align*}
    s(n) :=
    \begin{dcases}
        (-1)^{k+1} & \text{if } 24n=(6k+1)^2-25 \text{ or }24n=(6k+1)^2-1 \text{ for some } k\in\mathbb{Z}, \\
        %\text{if there is a }k \in \Z \text{ such that }
        0 & \text{otherwise}.
    \end{dcases}
\end{align*}
%\begin{equation}
%\begin{aligned}
%    \sum_{m \geq 1} s(m)q^m &:= \sum_{n=-\infty}^{\infty} (-1)^{n+1} q^{(6n+1)^2-25}+\sum_{n=-\infty}^{\infty}(-1)^{n+1}q^{(6n+1)^2-1}\\ & \hspace{.11cm} =  -q^{-24}+2q^{24}+q^{48}-q^{96}+\cdots.
%\end{aligned}
%\end{equation}
\begin{rem}
It is an easy exercise to confirm $s(n)$ is well-defined.
\end{rem}
\noindent Then we have the following result.
\begin{thm} \label{thm: 1.2} If $n \geq 1$, then 
 % \item \emph{if} $n \not\equiv 24 \bmod{25}$,   $$c(n) = \frac{1}{n} \sum_{k \in \Z}(-1)^k h_1(24n-(6k+1)^2)+\frac{s(24n)}{n}$$
\begin{align*} \label{eqn: thm 1.2}
   c(n) = \frac{s(n)}{n}+ \frac{1}{n} \sum_{k \in \Z} \Big[ (-1)^k h_1(24n-(6k+1)^2)+ (-1)^k h_2(24n-(6k+1)^2)\Big].
\end{align*}
\end{thm}
\begin{rem}
The formula in Theorem 1.2 bears a strong resemblance to another well-known expression for the coefficients of $j$. Work of Kaneko \cite{kaneko1999traces} shows for $n \geq 1$ that
\begin{align}
    c(n) = \frac{1}{n}\sum_{r \in \Z}\left[ \textbf{t}(n-r^2)-\frac{(-1)^{n+r}}{4}\textbf{t}(4n-r^2)+\frac{(-1)^{r}}{4}\textbf{t}(16n-r^2)
    \right],
\end{align}
where $\textbf{t}$ are traces of singular moduli, i.e. the sums of the $j$-invariants of elliptic curves with complex multiplication. In view of the similarity of these expressions, it is natural to wonder whether Theorem 1.2 suggests a deep connection between partitions and traces of singular moduli. 
%The similarity of these expressions suggests the existence of a deeper relationship between our $c(n)$-decomposition and the traces of singular moduli.
%is a specialization of a combinatorially-generated Jacobi form. 
\end{rem}
%COME BACK HERE TO RESTRUCTURE

In \cite{andrews2013concave}, Andrews related $spt(n)$ to a number of other combinatorial and number-theoretic functions. One connection of particular interest is the relationship of $spt$ to strongly unimodal sequences. 
%Andrews related $spt(n)$ to strongly unimodal sequences in  \cite{andrews2013concave}
We ask whether this relationship reveals deeper connections to the $j$-function and representation theory.
%look at another combinatorial function encapsulating $spt(n)$ which gives statistics of unimodal sequences.
%we look at a different function where $spt(n)$ appears alongside another object of combinatorial interest: unimodal sequences.

A sequence of integers $\{a_k\}_{k=1}^s$ is a strongly unimodal sequence of size $n$ if $\sum_{k=1}^{s}a_k=n$ and for some $r$ it satisfies $0<a_1<a_2<\cdots<a_r>a_{r+1}>a_{r+2}>\cdots>a_s>0$. The rank of $\{a_k\}_{k=1}^s$ is $s - 2r + 1$, the number of terms after the maximal term minus the number of terms preceding it.
%These are tied together via a new object of study first defined in \cite{zagier2010quantum} called the quantum modular form. 
%Motivating our interest is the fact that $spt$ surprisingly appears as part of a quantum modular form $\phi(x)$ related to Kontsevich's ``strange'' function $F(q)$ \cite{bryson2012unimodal}. % TODO: rethink this reference, use a more appropriate one. 
%This connection appears through 
The function $U(t;q)$ counts specific types of strongly unimodal sequences \cite{bryson2012unimodal}. For $t = -1$, %the coefficient $u^*(n)$ in
\begin{align*}
    U(-1;q) = \sum_{n \ge 1} u^*(n)q^n = q + q^2 - q^3 - 2 q^5 + 2 q^6 + \cdots,
\end{align*}
where $u^*(n)$ is the difference of the number of even-rank strongly unimodal sequences of size $n$ and the number of odd-rank strongly unimodal sequences of size $n$. %Bryson et al. \cite{bryson2012unimodal} proved the following identity when $t = -1$.
Andrews proved in \cite{andrews2013concave} that
\begin{align}
    U(-1;q) = -\sum_{n \ge 1}spt(n)q^n + 2A(q),
    \label{eqn: U spt and A}
\end{align} 
where
\begin{align*}
    A(q) = \sum_{n \ge 1} a(n)q^n := \frac{1}{(q;q)_\infty}\sum_{n=1}^\infty \frac{(-1)^{n-1}n q^{\frac{n^2+n}{2}}}{1-q^n} = q + q^2 - q^3 + q^4 - q^5 + 4 q^6 + \cdots. % write first few terms
\end{align*}

%\begin{thm} [Theorem 1.1, \cite{bryson2012unimodal}]
%If $q$ is a root of unity, then $U(-1;q)=F(q^{-1}).$
%\end{thm}

 It is natural to ask what $A(q)$ is counting. We find that $A(q)$ is the generating function for a partition statistic that we call the ``signed triangular weight" of a partition, a result which is of independent interest. Given a partition $\lambda \vdash N$, where we write the size of the partition as $|\lambda|:=N$, let $n_\lambda$ be the maximal number such that $\lambda$ contains parts of size $1,2,\ldots, n_\lambda$. Letting $m_k$ denote the number of times that the part $k$ appears in $\lambda$, we define the signed triangular weight of $\lambda$ to be $t_s(\lambda) := \sum_{k=1}^{n_\lambda}(-1)^{k-1}km_k$. If $\lambda$ does not contain a part of size $1$, then let $t_s(\lambda) = 0$.
\begin{exmp}
Consider $\lambda=\{1,2,2,3,4,5,5,8\}$. Then $\lambda \vdash 30$, $n_\lambda=5$, and
$$
%t(\lambda)=1+2\cdot2+3+4+2\cdot5=22,
%\text{ \ \ and \ \ }
t_s(\lambda)=1\cdot1-2\cdot2+3\cdot1-4\cdot1+5\cdot2=6.
$$
\end{exmp}
\noindent We prove the following result relating  $t_s(\lambda)$ for all partitions $\lambda$ of all positive integers to the series $A(q)$.
\begin{thm}
    \label{thm: 1.3 triangular weight}
    The following $q$-series identity is true:
    \begin{align*}%\label{eqn:thm 1.2}
        %\frac{1}{(q;q)_\infty} \sum_{n \geq 1} \frac{(-1)^{n-1} n q^{\frac{n(n+1)}{2}}}{1-q^n}
        A(q) &= \sum_{\lambda} t_s(\lambda) q^{|\lambda|}.
    \end{align*}
\end{thm}
\noindent From this, we may conclude that $a(n)=\sum_{ |\lambda|=n } t_s(\lambda)$. 
%Armed with this new power series $A(q)$, we see that $U(-1;q)$ involves $spt$ and our new statistic $t_s(\lambda)$. 
%
%As a corollary, we find congruences relating the coefficients of $A$ and $U(-1;q)$ due to the families of congruences found by Ono and SOMEONE ELSE AHLGREN? LOVEJOY? NEED TO CITE THEM for $spt$ modulo prime powers.
%
%CONGRUENCES
%
%combined with action of hecke operators, we can see that U(q), the study of unimodal seqs, is int erms of ptn and A. so we have expressions for these polys in J in terms of spt and A (a couple choices). and in view of the theorem giving action of hecke operators, we have a combinatorial desc of coefficients for all these polys in j in terms of B
%
%for every $\ell$, there is a formula involving derivative of j and polynomial of j in terms of these gadgets In the special case for $\ell = 5$, this gives rise to a formula for the coefficients of $j$.
%
Given this relationship, the $spt$ congruence given in (\ref{eqn: spt cong for powers}) immediately implies the following result.
%Following (\ref{eqn: spt cong for powers}), we immediately have the following congruences relating the quantities $u^*(n)$ and $a(n)$.
\begin{cor} \label{cor 1.4: congruences}
    If $\ell \ge 5$ is prime, $\left(\frac{-n}{\ell} \right) = 1$, and $m \geq 1$, then 
     \begin{align*}
         u^*\left( \frac{\ell^{2m}n-1}{24}\right) \equiv 2a\left( \frac{\ell^{2m}n-1}{24}\right) \pmod{\ell^m}.
     \end{align*}
\end{cor}

Combining our explicit expression for the action of the Hecke operator $T(25)$ in Theorem~\ref{thm: 1.1 Hecke} and our combinatorial expressions for $c(n)$, we arrive at new expressions for the coefficients of $j(\tau)$ in terms of $p(n)$ and the coefficients of $a(n)$ and $u^*(n)$. For ease of notation, we define the functions
\begin{equation}
\begin{aligned}
    g_1(24n-1) &:=-\frac{12}{5}u^*(25n-1)+\frac{24}{5}a(25n-1)+5(24n-1)p(25n-1)\\
    &\hspace{0.5cm} + \mu_n \cdot  \left(-\frac{12}{5}u^*(25n-1)+\frac{24}{5}a(25n-1)+5(24n-1)p(n)\right), \\
%\end{align*}
%    \begin{center}
%        and
%    \end{center}
%\begin{align*}
    g_2(25(24n-1)) &:= -12u^*(n)+24a(n)+(24n-1) p(n),
\end{aligned}
\end{equation}
where as in (\ref{h coefficients}), $g_1(m)=0$ if $m \not\equiv 23 \bmod{24}$ and $g_2(m)=0$ if $m \not\equiv 23 \bmod{24}$ and $m \not\equiv 0 \bmod{25}$.
%Note that $g_1$ and $g_2$ are identical to the functions $h_1$ and $h_2$, respectively, defined for Theorem \ref{thm: 1.2}, except for replacing $spt(n) = -u^*(n) + 2a(n)$.
\begin{cor}
     \label{cor: 1.5 c(n)}
     %For $n \geq 1$, let $c(n)$ denote the $n$th coefficient of the $q$-expansion of the $j$-function. Then
     If $n \ge 1$, then
%\emph{\begin{enumerate}
    %\item \emph{if} $n \not\equiv 24 \bmod{25}$,  
   %$$
  % c(n) = \frac{1}{n} \sum_{k \in \Z}(-1)^k g_1(24n-(6k+1)^2)+\frac{s(24n)}{n}
  % $$
  % \item \emph{if} $n \equiv 24 \bmod{25}$,
   \begin{align*}
   c(n) = \frac{s(n)}{n} + \frac{1}{n} \sum_{k \in \Z}\Big[ (-1)^kg_1(24n-(6k+1)^2)+ (-1)^kg_2(24n-(6k+1)^2)\Big].
   \end{align*}
%\end{enumerate}}
\end{cor}
\begin{exmp}
Using our result, we find the following identities:
\begin{align*}
    c(1) = 196884=s(24)+\frac{168a(1)-84u^*(1)+161p(1)}{5}+\frac{24}{5}a(24)-\frac{12}{5}u^*(24)+115p(24)
\end{align*}
\begin{align*}
    c(2) = 21493760 &= \frac{1}{2}\Big(s(48)-\frac{168a(1)-84u^*(1)+161p(1)}{5}+\frac{14a(2)-7u^*(2)+329p(2)}{5}\\
    & \hspace{.65cm} -\frac{24}{5}a(24)+\frac{12}{5}u^*(24)-115p(24)+\frac{24}{5}a(49)-\frac{12}{5}u^*(49)+235p(49)\Big).
\end{align*}
\end{exmp}

%\begin{rem}
    %In view of the combinatorial structure involving $spt(n)$, $p(n)$, unimodal sequences, and our new partition statistic encoded by $A(q)$, the similarity of the above formulas for $c(n)$ to Kaneko's suggest a hidden structure in the monster Lie algebra, as described by the $j$-function, that deepens the relationship of partitions to representation theory.
%\end{rem}

%The form of Theorem \ref{thm: 1.2} is analogous to the forms of the recurrences for Kronecker class number relations and the Eichler-Selberg trace formula, as investigated by Michael Mertens in his Ph.D. thesis \cite{mertens2014mock}. Furthermore, the form of minimal-divisor sum for $n$ in the Kronecker class number relations is another form of the unsigned triangular weight we define for partitions of $n$ \cite{hirschhorn2005partitions}.

%Further, Conway and Norton \cite{conway1979monstrous} noted that the graded dimension of $V^{\natural}$ is the graded trace of the identity element of the monster, and asked if the graded traces for other elements $g$ would also correspond to the coefficients of a $q$-series that would be of interest. Borcherds answered this question affirmatively as well, confirming a conjectured list of Hauptmoduln given by Conway and Norton. Here, we ask if there are other combinatorial interpretations for these coefficients that generalize the ideas started here.
\begin{question}
Are the combinatorial interpretations of the coefficients of the $j$-function in Theorem~\ref{thm: 1.2} and Corollary~\ref{cor: 1.5 c(n)} glimpses of hidden structure of the monster module? In particular, do $spt(n)$, $u^*(n)$, and $a(n)$ play roles in representation theory?
\end{question}

\begin{rem}
    After this paper was submitted, Toshiki Matsusaka informed the authors \cite{matsusaka} that he has obtained further similar results along these lines which frame the $spt$ function in terms of a weakly holomorphic Jacobi form. This structure also provides a connection to the formulation by Kaneko \cite{kaneko1999traces} of the $j$-function's coefficients using traces of singular moduli.
\end{rem}

This paper is organized as follows. In Section 2, we investigate the specific harmonic Maass form $\mscr{M}(\tau)$ and derive an expression for the action of the Hecke operator on its holomorphic part. To do this, we study canonical families of polynomials in $j(\tau)$ and explore the relationship of modular forms to modular functions on $\SL_2(\Z)$.
In Section 3, we prove Theorem~\ref{thm: 1.3 triangular weight}.
%examine the series $U(-1;q)$ and
%we look at the quantum modular form $\phi(x)$ that connects $spt(n)$ to unimodal sequences and the series $A(q)$. Here,
%define a partition statistic for $A(q)$ in order to combinatorially ground our form as well as exploit the simultaneous appearance of $spt$ in both $\mscr{M}(\tau)$ and $U(-1;q)$. 
In Section 4, 
%we combine these two approaches to first give congruences modulo primes between the coefficients of $A(q)$ and $U(-1;q)$, and then to 
we prove Theorem~\ref{thm: 1.2} and Corollary~\ref{cor: 1.5 c(n)}.
%derive two different combinatorial expressions for the coefficients of $j(\tau)$. 
%%%%%%%%%%%%%%%%%%%%%%%%%%%%%%%%%%%%%%%%%%%%%%%%%%%%%%
%%%%%%%%%%%%%%%%%%%%%%%%%%%%%%%%%%%%%%%%%%%%%%%%%%%%%%
%%%%%%%%%%%%%%%%%%%%%%%%%%%%%%%%%%%%%%%%%%%%%%%%%%%%%%

\section{Harmonic Maass Forms}

%%%%%%%%%%%%%%%%%%%%%%%%%%%%%%%%%%%%%%%%%%%%%%%%%%%%%%
%%%%%%%%%%%%%%%%%%%%%%%%%%%%%%%%%%%%%%%%%%%%%%%%%%%%%%
\subsection{Preliminaries}
%2.1 Preliminaries
% Definitions: fourier expansion, properties of xi operator, also zweger's mu function, where the point is that it gives infinite supply of harmonic maass forms. recall the functional equation and whatever else you need.

To motivate our study and to ground the methods used here, we begin by introducing the harmonic Maass form of interest for this paper. 
%We will assume that the reader is familiar with modular forms and harmonic Maass forms (see, e.g. \cite{ono2004web} and \cite{ono2009unearthing}).
Recall that a weakly holomorphic modular form for a congruence subgroup $\Gamma$ of $\SL_2(\Z)$ is a function that is holomorphic on $\H$, whose poles, if any, are supported on the cusps of $\Gamma \backslash \H$, and which satisfies the corresponding modularity properties for its weight. If $f$ is a weakly holomorphic modular form of weight $k$ for $\Gamma$ and Nebentypus $\chi$, we write $f \in M_k^!(\Gamma, \chi)$. 

Likewise, a smooth function $f\colon \H \rightarrow \C$ is a harmonic Maass form of weight $k$ for $\Gamma$ and $\chi$ if it satisfies the standard modular transformation laws, is annihilated by the harmonic Laplacian $\Delta_k$, and has at most growth-order 1 exponential growth at each cusp on $ \Gamma\backslash \H$.
We denote the vector space of harmonic Maass forms of weight $k$ for $\Gamma$ and $\chi$ as $H_k(\Gamma, \chi)$. 
%Following the notation of \cite{ono2011spt}, we define the generating functions of Andrews' smallest parts function $spt(n)$ as %and the partition function $p(n)$ as
%\begin{align}
   % \mc{S}(q) := \sum_{n \ge 1} spt(n)q^{24n-1}.
%\end{align}
%Let $\eta(\tau) := q^{1/24}\prod_{n \ge 1} $

Recalling the definitions of $\mathscr{P}(q)$ and $\mathcal{S}(q)$ in (\ref{eqn: partition generator}) and (\ref{eqn: spt generator}), we define $\mscr{M}(\tau)$ following \cite{ono2011spt} as
\begin{align} 
\label{eqn: script M whole thing}
    \mscr{M}(\tau) := \mathcal{S}(q) + \frac{1}{12}q \frac{d}{dq}\mscr{P}(q) - \frac{i}{4\pi \sqrt{2}} \cdot \int_{-\bar{\tau}}^{i\infty} \frac{\eta(24z)}{[-i(z + \tau)]^{3/2}}\; dz.
\end{align}
%where $\eta(\tau)$ is Dedekind's eta-function, %$$\mathcal{S}(q) := \sum_{n=1}^\infty spt(n)q^{24n-1}$$ is the generating function for Andrews' spt function, and $$\mscr{P}(q) := \sum_{n=0}^\infty p(n)q^{24n-1}$$ is the generating function for the partition function. 
By Theorem~2.1 of \cite{ono2011spt}, $\mscr{M}(\tau) \in H_{3/2}\left(\Gamma_0(576), \chi_{12}\right)$, where $\chi_{12} := \left(\frac{12}{\cdot}\right)$. By $M^+(q)$ we denote the holomorphic part of $\mscr{M}(\tau)$. This may be expressed as
\begin{align*}
    M^+(q) := \mathcal{S}(q) + \frac{1}{12}q \frac{d}{dq}\mscr{P}(q) = -\frac{1}{12}q^{-1} + \frac{35}{12}q^{23} + \frac{65}{6}q^{47} + \cdots.
\end{align*}

%%%%%%%%%%%%%%%%%%%%%%%%%%%%%%%%%%%%%%%%%%%%%%%%
%%%%%%%%%%%%%%%%%%%%%%%%%%%%%%%%%%%%%%%%%%%%%%%%
\subsection{The Hecke Action}
% 2.2 Hecke leveraging
% Main lemma: M_ell and Q_ell are weakly holomorphic modular forms of the right weight in the right spaces. Proof invovles using eigenvalues of eta and eta^3 under the Hecke algebra. and M_ell = M|T(ell^2) - whatever*M, and Q_ell = whatever. This kills the nonholomorphic part because of the stuff defined in 2.1.
To understand the action of the Hecke operator on $M^+$, we will need the following result that produces a weakly holomorphic modular form involving $M^+(\tau)\mid T(\ell^2)$. We produce this modular form via the following result.
\begin{lem} \label{lemma 1}
If
\begin{align*}
    M_\ell(\tau) := M^+(\tau) \mid T(\ell^2) - \left(\frac{3}{\ell}\right) (1+\ell)M^+(\tau),
\end{align*}
then $M_\ell(\tau) \in M_{3/2}^!(\Gamma_0(576), \chi_{12}).$
\end{lem}

\begin{proof}
Up to a constant, the nonholomorphic part of $\mscr{M}(\tau)$ is the period integral for $\eta(24\tau)$. Write $\tau = x + iy$ for $x,y \in \R$. Under the action of the differential operator $\xi_k := 2iy^{k}\frac{\overline{\partial}}{\partial \overline{\tau}}$, we have $\xi_{3/2}(\mscr{M}) = -\frac{1}{8\pi} \eta(24\tau)$. Note that $\eta(24\tau)$ is an eigenform for Hecke operators of weight $1/2$ with eigenvalue $\chi_{12}(\ell)(1+\ell^{-1}) = (\frac{3}{\ell})(1+\ell^{-1})$. If we define 
\begin{align*}
    \mscr{M}_\ell(\tau) := \mscr{M}(\tau) \mid T(\ell^2) - \left(\frac{3}{\ell}\right) (1+\ell)\mscr{M}(\tau),
\end{align*}
we observe that
\begin{align}
    \xi_{3/2}\left[\mscr{M}(\tau) \mid T(\ell^2) - \left(\frac{3}{\ell} \right)(1+\ell)\mscr{M}(\tau)\right]=0.
    \label{xi kills M_ell}
\end{align}
Here we have used the commutativity relation
\begin{align*}
    \xi_k(f \mid_k T(\ell^2))=\ell^{2k-2}(\xi_k f)\mid_{2-k}T(\ell^2)
\end{align*}
for half-integral weight harmonic Maass forms given in Proposition 7.1 of \cite{bruinier2010heegner}. 
Since the Hecke algebra preserves modularity, $\mscr{M}(\tau)\mid T(\ell^2) \in H_{3/2}(\Gamma_0(576), \chi_{12})$. By (\ref{xi kills M_ell}), $\mscr{M}_\ell(\tau)$ is in the kernel of $\xi_{3/2}$ and is therefore holomorphic on the upper half plane. Since the action of the Hecke and $\xi$ operators are linear and thus split over the holomorphic and nonholomorphic parts of $\mscr{M}(\tau)$, the same result holds for $M^+(\tau)$. In particular, $M_\ell(\tau) \in M_{3/2}^!(\Gamma_0(576), \chi_{12}).$
\end{proof}

%Let $f \in H_{2-k}(\Gamma_0(N),\chi)$, and suppose that $\xi_{2-k}(f) \in S_k(\Gamma_0(N), \chi)$ is a Hecke eigenform over $T(\ell^2)$ with eigenvalue $\lambda(\ell^2)=\ell^{2-2k}\chi(\ell^2)$. Then 
%\begin{align}
    %f \mid T_{2-k}(\ell^2)- \ell^{2-2k}\lambda(\ell^2)f \in  M_{2-k}^!(\Gamma_0(N),\chi).
%\end{align}
%\begin{proof}
    %By Proposition 7.1 of \cite{bruinier2007heegner}, the operator $\xi$ and the Hecke operator $T(\ell^2)$ commute in the following way:
    %\begin{align}
        %\xi_{2-k}(f \mid T_{2-k}(\ell^2))&= \ell^{2-2k}(\xi_{2-k}f) \mid T_{k}(\ell^2) \\
        %&= \ell^{2-2k}\lambda(\ell^2) \xi_{2-k}(f),
    %\end{align}
    % need to specify how rearranging this gives the result.
   % since $\xi_{2-k}(f) \mid T(\ell^2)= \lambda(\ell^2)\xi_{2-k}(f)$ by assumption. By the linearity of $\xi_{2-k}$,
    %\begin{align}
        %\xi_{2-k}\left(f \mid T_{2-k}(\ell^2) - \ell^{2-2k}\lambda(\ell^2)f \right) = 0.
   % \end{align}
    %Since $\xi$ annihilates precisely the holomorphic functions on the upper half plane, we may deduce that the argument $f \mid T_{2-k}(\ell^2) - \ell^{2-2k}\lambda(\ell^2)f$ is holomorphic. Combining this with the fact that $f$ and $f \mid T_{2-k}(\ell^2)$ satisfy the M\"obius transformation law for $\Gamma_0(N)$ and $\chi$ with weight $2-k$, we have the desired result.
%\end{proof}

%%%%%%%%%%%%%%%%%%%%%%%%%%%%%%%%%%%%%%%%%%%%%%%%%%%%%%%
%%%%%%%%%%%%%%%%%%%%%%%%%%%%%%%%%%%%%%%%%%%%%%%%%%%%%%%
\subsection{Canonical polynomials in $j(\tau)$}
%%%%%%%%%%%%%%%%%%%%%%%%%%%%%%%%%%%%%%%%%%%%%%%%%%%%%%%
%%%%%%%%%%%%%%%%%%%%%%%%%%%%%%%%%%%%%%%%%%%%%%%%%%%%%%%
We show that the set of all $B_m(j(\tau))$ form a convenient $\C$-basis for the ring of weakly holomorphic modular functions on $\SL_2(\Z)$ as a $\C$-vector space. Recall that the ring of weakly holomorphic modular functions on $\SL_2(\Z)$ is precisely the ring of complex polynomials in $j(\tau)$, i.e. $M^!_0(\SL_2(\Z)) = \C[j(\tau)]$ \cite[Theorem~2.8]{apostol2012modular}. As defined in (\ref{eqn: B_m defined}), we have
\begin{align*}
    B_1(x) &= 1, \\
    B_2(x) &= x - 745, \\
    B_3(x) &= x^2 - 1489x + 357395.
\end{align*}
From these first few examples, the set of $B_m(x)$ appears to form a $\C$-basis for the polynomial ring $\C[x]$ as a $\C$-vector space, and hence the set of $B_m(j(\tau))$ would form a $\C$-basis for $M^!_0(\SL_2(\Z))$. In the following lemma, we show that this is indeed the case. To do so, we define the function
\begin{align}
    \alpha(q) := \frac{(q;q)_\infty}{-q \frac{d}{dq} j(\tau)} = q + O(q^2).
\end{align}

\begin{lem} \label{lemma: B_n basis}
    If $f(\tau)$ is a weakly holomorphic modular function on $\SL_2(\Z)$ and is of the form 
    \begin{align} \label{eqn: f in lemma for B_m}
        f(\tau) = \alpha(q)\left( \sum_{n \gg -\infty }^{-1} t(n)q^n \right) + O(q),
    \end{align}
    then
    \begin{align*}
        f(\tau) = \sum_{n \gg -\infty }^{-1} t(n)B_{-n}(j(\tau)).
    \end{align*}
\end{lem}

\begin{rem}
    The above lemma gives a clean formulation for modular functions $f$ of the form given in (\ref{eqn: f in lemma for B_m}) when the principal part of $f/\alpha$ is known.
\end{rem}

\begin{proof}[Proof of Lemma~\ref{lemma: B_n basis}]
    For each $m \ge 0$, note that there exists a unique weakly holomorphic modular function $j_m(\tau)$ on $\SL_2(\Z)$ such that $j_m(\tau) = q^{-m} + O(q)$. The Faber polynomials $J_n(x)$ are the coefficients of the generating function
    \begin{align*}
        \sum_{n=0}^\infty J_n(x) q^n := \frac{E_4^2(\tau)E_6(\tau)}{\Delta(\tau)} \cdot \frac{1}{j(\tau) - x} = 1 + (x -744) q + \cdots.
    \end{align*}
    By Corollary 4 in \cite{asai1997zeros}, $J_n(j(\tau)) = j_n(\tau)$ for all $n \ge 0$. By comparing the generating functions for $J_n(x)$ and $B_n(x)$ and using the identity (\ref{eqn: B_m defined}), we see that 
    \begin{align*}
        \alpha(q) \sum_{n \ge 0} J_n(x) q^n = \alpha(q) \cdot \frac{E_4^2(\tau)E_6(\tau)}{\Delta(\tau)} \cdot \frac{1}{j(\tau) - x} = (q;q)_\infty \cdot \frac{1}{j(\tau) - x} = \sum_{n \ge 1} B_n(x) q^n.
    \end{align*}
    Since $\alpha(q) = q + O(q^2)$, we compare coefficients and deduce that for each $n \ge 1$, 
    \begin{align*}
        \alpha(q)J_{n}(j(\tau)) = B_n(j(\tau)) = \alpha(q)q^{-n} + O(q).
    \end{align*}
    %Since $J_{n}(j(\tau)) = q^{-n} + O(q)$, 
    And hence we can conclude that
    \begin{align*}
        f(\tau) &= \alpha(q)\left( \sum_{n \gg -\infty }^{-1} t(n) q^n \right) + O(q) = \sum_{n \gg -\infty}^{-1} t(n)B_{-n}(j(\tau)).
    \end{align*}
\end{proof}

\subsection{Proof of Theorem~1.1}
% subsubsection 3.2.1. 

%The following identity gives the relationship between the $\xi$ operators and the Hecke operator $T(\ell^2)$. With a slight manipulation, it gives \Cref{lemma 7.4}.
% Prop 7.1 of the Heegner divisors paper:

Note that we may write 
\begin{align}
    M_\ell(z) = -\frac{\ell}{12} q^{-\ell^2} + \left(\frac{3}{\ell} \right) \frac{\ell}{12} q^{-1} + \sum_{\mathclap{\substack{n \ge 23 \\ n \equiv 23 \bmod{24}} }} a_{\ell}(n) q^n,
\end{align}
where we observe that, since $\ell^2 \equiv 1 \bmod{24}$, the nonzero coefficients of $M_\ell$ are supported on integral exponents that are 23 $\bmod{ \ 24}$. Following this, we define 
\begin{align}
    F_\ell(24\tau) := \eta^{\ell^2}(24\tau)M_\ell(\tau).
\end{align}
By Lemma \ref{lemma 1}, it is immediate that $F_\ell(24\tau)$ is a weakly holomorphic modular form of weight $\frac{\ell^2+3}{2}$ over $\Gamma_0(576)$ with trivial Nebentypus.
%In order to proceed, we will need the following result concerning the level of $F_\ell$ to ensure that we may indeed characterize our final answer in terms of $j(\tau)$. 
\begin{comment}
    \begin{thm}[Theorem~2.2, \cite{ono2011spt}]
    Suppose $\ell \geq 5$ is prime and define
    \begin{align*}
        F_\ell(\tau):=\eta^{\ell^2}(\tau)M_\ell(\tau/24).
    \end{align*} Then $F_\ell(z)$ is a weight $\frac{\ell^2+3}{2}$ holomorphic modular form on $\SL_2(\Z)$.
    \end{thm}
\end{comment}
In fact, by Theorem~2.2 in \cite{ono2011spt}, $F_\ell(\tau)$ is a weight $\frac{\ell^2+3}{2}$ holomorphic modular form on $\SL_2(\Z)$. We recall that the proof makes use of the observation that $F_\ell \in \Z[[q^{24}]]$ by construction, and that the behavior of $F_\ell$ under the matrix $S=\begin{psmallmatrix}
0 & -1 \\
1 & 0
\end{psmallmatrix}$ can be determined using a result of Bringmann in \cite{bringmann2008explicit} which gives that $\mscr{M}(\tau)$ is an eigenform of the Fricke involution.

\subsubsection{Getting to Weight 0}

Now that we have a holomorphic modular form of weight $\frac{\ell^2+3}{2}$ on all of $\SL_2(\Z)$, we will leverage this information, along with some properties of the Eisenstein series $E_{14}$ and the $j$-function, to produce a closed formula for the Hecke action. We first note that $\frac{\ell^2+3}{2} \equiv 2 \bmod{12}$, and that likewise so is $E_4^2(\tau)E_6(\tau)$. To make use of this seemingly innocuous fact, define $\delta_\ell := \frac{\ell^2 - 1}{24}$ and note that
\begin{align}
    G_\ell(\tau) := E_4^2(\tau)E_6(\tau) \Delta^{\delta_\ell - 1}(\tau) = q^{\delta_\ell - 1} + \ldots \in M_{\frac{\ell^2 + 3}{2}}(\SL_2(\Z)).
\end{align}
Since we now have another modular form of the same weight on $\SL_2(\Z)$, we would like to prove that their quotient, $F_\ell(\tau)/G_\ell(\tau)$, is a weakly holomorphic modular function on $\SL_2(\Z)$, which, coupled with our preceding characterization of the Faber polynomials, will allow for a unique expression of the quotient as a polynomial in $j(\tau)$. 

\begin{lem}
    The function $F_\ell(\tau)/G_\ell(\tau)$ is a polynomial in $j(\tau)$.
\end{lem}
\begin{proof}
    By construction, $G_\ell$ has a zero of degree 2 at $e^{2\pi i/3}$, a simple zero at $i$, and no other zeros in the fundamental domain $\mathcal{F}$ of $\SL_2(\Z)$.
    
    %Since the weight of $F_\ell$ is $\frac{\ell^2 + 3}{2} \equiv 2 \bmod{12}$, we may use the valence formula to deduce that $F_\ell$ vanishes at $e^{2\pi i/3}$ with multiplicity at least $2$ and vanishes at $i$ with multiplicity at least $1$. Therefore the quotient $F_\ell/g_\ell$ has no poles in $\mathcal{F}$. Hence we may deduce that $F_\ell/g_\ell$ is a weakly holomorphic modular form of weight 0 on $\SL_2(\Z)$.
    % NEW AND IMPROVED:
    Since the weight of $F_\ell$ is $k=(\ell^2+3)/2 \equiv 2 \bmod{12}$, we apply the transformation law under $S = \begin{psmallmatrix} 0 & -1 \\ 1 & 0\end{psmallmatrix}$ to get that $F_\ell(-1/i) = i^kF_\ell(i) = -F_\ell(i)$, hence $F_\ell(i) =0$. Similarly, applying the transformation law under $\gamma = \begin{psmallmatrix} 0 & -1 \\ 1 & 1 \end{psmallmatrix}$ yields $F_\ell(e^{2\pi i / 3}) = 0$. Differentiating both sides of $F_\ell( \gamma \tau) = (\tau + 1)^k F_\ell(\tau)$ and letting $\tau = e^{2\pi i /3}$ gives that $\frac{d}{d\tau} F_\ell(\tau )|_{\tau = e^{2\pi i/3}} = 0$. Hence $F_\ell$ vanishes at $e^{2\pi i/3}$ with order at least 2. Therefore the quotient $F_\ell/G_\ell$ has no poles in $\mathcal{F}$, and we may deduce that $F_\ell/G_\ell$ is a weakly holomorphic modular form of weight 0 on $\SL_2(\Z)$.
    Since the modular functions on $\SL_2(\Z)$ are precisely the polynomials $\C[j(\tau)]$, we may conclude that $F_\ell/G_\ell$ is a polynomial in $j(\tau)$. 
\end{proof}
It remains to construct this polynomial in $j(\tau)$.
%The rest of the proof of Theorem~\ref{thm: 1.1 Hecke} will rely on our ability to explicitly construct this polynomial. 
Using the modular functions $B_{\delta_\ell}(j(24\tau))$, we arrive at the following conclusion:
\begin{align*}
    \frac{F_\ell(\tau)}{G_{\ell}(\tau)} &= \frac{\eta(\tau)^{\ell^2}M_{\ell}(\tau/24)}{E_4^2(\tau)E_6(\tau)\Delta^{\delta_\ell-1}(\tau)} \\
    &= \frac{(q;q)_\infty}{-q \frac{d}{dq}j(\tau)} q^{1/24}M_{\ell}(\tau/24) \\
    &= \frac{(q;q)_\infty}{-q \frac{d}{dq}j(\tau)} \left[ -\frac{\ell}{12} q^{-\delta_\ell} + \left(\frac{3}{\ell}\right)  \frac{\ell}{12} + O(q)  \right] \\
    &= \alpha(q) \left[ -\frac{\ell}{12} q^{-\delta_\ell} + \left(\frac{3}{\ell}\right)  \frac{\ell}{12} + O(q) \right] \\
    &= - \frac{\ell}{12} B_{\delta_\ell}(j(\tau)),
\end{align*}
where the last equality follows from Lemma \ref{lemma: B_n basis}. Hence we may rearrange to get the expression
\begin{align}
    M_\ell(\tau/24) &= \frac{F_\ell(\tau)}{\eta^{\ell^2}(\tau)}
    = - \frac{\ell}{12} \frac{E_4^2(\tau)E_6(\tau)}{\Delta(\tau)}  \eta^{-1}(\tau) B_{\delta_\ell}(j(\tau)).
\end{align}
Sending $\tau \mapsto 24\tau$ and using the fact that $\mscr{P}(q) = \eta^{-1}(24\tau)$, 
\begin{align*}
    M_\ell(\tau) &= \mscr{P}(q) \left(\frac{E_4^2(24\tau)E_6(24\tau)}{\Delta(24\tau)}\right) \left[ - \frac{\ell}{12} B_{\delta_\ell}(j(24\tau)) \right].
\end{align*}

We can finally conclude that the action of the Hecke operator $T(\ell^2)$ is
\begin{align*}
    M^+(\tau) \mid_{3/2} T(\ell^2) = \left(\frac{3}{\ell} \right)(1+\ell)M^+(\tau) - \frac{\ell}{12} \mscr{P}(q)\cdot B_{\delta_\ell}(j(24\tau)) \cdot \frac{E_4^2(24\tau)E_6(24\tau)}{\Delta(24\tau)},
\end{align*}
concluding the proof of Theorem 1.1. \hspace*{\fill}$\square$

%\section{Unimodal Sequences and Partition Statistics}
%Recall that in the series $U(-1;q) = \sum_{n\ge1} u^*(n)q^n$, the coefficient $u^*(n)$ represents the difference of the number of even-rank strongly unimodal sequences of size $n$ and the number of odd-rank strongly unimodal sequences of size $n$. Andrews \cite[Theorem~1]{andrews2013concave} proves the following relation:
%\begin{align}
%    U(-1;q) &:= -\sum_{n \ge 1} spt(n)q^n  + 2A(q), \label{eqn:U spt A}
%\end{align}
%where
%\begin{align}
%    A(q) &= \frac{1}{(q;q)_\infty} \sum_{n\geq 1}\frac{(-1)^{n-1}nq^{\frac{n^2+n}{2}}}{1-q^n}.
%\end{align}

%%%%%%%%
\begin{comment}
In \cite{bryson2012unimodal}, the authors prove that $\phi(x):=e^{2\pi i x}\cdot U(-1;e^{2\pi i x})$ is a quantum modular form of weight $3/2$ on $\SL_2(\Z)$; in particular, it satisfies the conditions $\phi(x+1)=\phi(x)$ and 
\begin{align}
    \phi(x)+(-ix)^{3/2}\phi(-1/x)=h(x),
\end{align}
where
\begin{align*}
    h(x):=\frac{\sqrt{3}}{2\pi i}\int_0^{i\infty}\frac{\eta(\tau)}{[-i(\tau+x)]^{3/2}}d\tau
    -\frac{i}{2}e^{(\pi i x)/6}(e^{2\pi i x};e^{2\pi i x})_\infty^2\cdot \int_0^{i\infty}\frac{\eta^3(\tau)}{[-i(\tau+x)]^{1/2}}d\tau.
\end{align*}
\end{comment}
%%%%%%%%%%
%Since $u^*(n)$ and $spt(n)$ both have combinatorial interpretations, one is motivated to ask whether the $q$-series $A(q)$ likewise can be explained combinatorially. 

\section{The Signed Triangular Weight}

%The fact that the motivation for studying $\mscr{M}(q)$ comes from its connection to Andrews' $spt$ function coupled with occurrence of the $spt$ function in our quantum modular form $\phi(x)$, one would like to have a combinatorial interpretation for the last sum in our expression for $U(-1;q)$ for which we can derive results via our study of the Hecke action.
%Since the coefficients of both $U(-1;q)$ and $\mc{S}(q)$ have combinatorial significance, w
%%%% ^I think the sentence preceding this subsubsection suffices for motivation

%We now present a surprising connection between a natural notion of weighting partitions and this final piece $A(q)$ of (\ref{eqn: U spt and A}).

In light of the connection between the generating functions of the Andrews $spt$-function and a particular class of unimodal sequences given in (\ref{eqn: U spt and A}) mediated by the series $A(q)$, we present the proof of Theorem 1.3.

\subsection{Proof of Theorem 1.3}
We begin by examining the summation \begin{align}
    \sum_{n \geq 1} \frac{q^{\frac{n(n+1)}{2}}}{1-q^n}.
\end{align}
By considering each summand to be of the form 
%For a single summand, we have
$q^{\frac{n(n+1)}{2}}(1+q^n+q^{2n}+q^{3n}+\cdots)$,
note that if we formally expand the above power series as $\sum_{m \ge 1} \alpha(m)q^m$, then $\alpha(m)$ counts the number of ways to choose integers $(n,k)$ with $n \ge 1$, $k \ge 0$ such that $m = T_n + kn$, where $T_n = n(n+1)/2$ denotes the $n$th triangular number.
%Recognizing that $\frac{n(n+1)}{2}=1+2+...+n = T_n$, the $n$th triangular number, we recognize that a term from this --- can only contribute to the coefficient of $q^m$ if $m=1+2+3+...+(n-1)+k\cdot n$, where $k$ is an integer at least 1. A partition-theoretic interpretation tells us that we are then considering partitions of $m$ of this form. 
Similarly, the coefficient $\beta(m)$ of $q^m$ in the formal expansion of
\begin{align}
    \sum_{n \geq 1} \frac{(-1)^{n-1} n q^{\frac{n(n+1)}{2}}}{1-q^n} := \sum_{m \ge 1} \beta(m)q^m
\end{align}
denotes a sum over all such pairs $(n,k)$, weighted by the parity and size of $n$. 

Multiplying the above series by the generating function $1/(q;q)_\infty$ for partitions then gives a formal power series
\begin{align}
    \frac{1}{(q;q)_\infty}\sum_{n \geq 1} \frac{(-1)^{n-1} n q^{\frac{n(n+1)}{2}}}{1-q^n} := \sum_{m \ge 1} \gamma(m) q^m
\end{align}
where $\gamma(m)$ runs over all partitions $\lambda \vdash m$ such that $\lambda$ contains a subpartition consisting of the parts $\{1,2,\ldots,n\}$ and also possibly $k$ more parts of size $n$, for $n \ge 1$ and $k \ge 0$, but weighting this count by the parity and size of $n$. \hspace*{\fill}$\square$

\section{Combinatorial interpretations of the coefficients of $j(\tau)$}

As we have now developed a variety of both combinatorial and number-theoretic objects, all of which are tied together by a class of polynomials in $j(24\tau)$, it is natural to ask if we may formalize and explicate this connection. To do this, we make use of both the standard definition of the Hecke operator on $q$-series expansions as well as the result of Theorem~1.1 in order to pull the functions $spt(n)$ and $p(n)$ through to the $j$-function. We restrict our attention to the case where $\ell=5$ since $\delta_\ell=1$ and $B_1(j(24\tau))=1.$ While at first glance it may seem as though we have removed $j$ from our expressions by looking at this case, we recall that \begin{align}
    -q \frac{d}{dq}j(24\tau) = \frac{E_4^2(24\tau) E_6(24\tau)}{\Delta(24\tau)} = q^{-24} - \sum_{n\geq 1} n  c(n) q^{24n}, 
\end{align}
where $c(n)$ is the $n$th coefficient of the $j$-function. Thus, we need only solve for the $c(n)'s$ in terms of the combinatorial information given by $M^+(\tau)$ to arrive at our final conclusions. 

\subsection{Proof of Theorem~1.2} Writing the $q$-series expansion for $M^+(\tau)$ out in terms of $spt(n)$ and $p(n)$, we arrive at
\begin{align*}
    M^+(\tau)=-\frac{1}{12}q^{-1}+\sum_{n\geq 1} \left[ spt(n)+\frac{(24n-1)}{12}p(n) \right] q^{24n-1}.
\end{align*}
For $n \geq 1$, we then write for ease
\begin{align}
    m(24n-1) := spt(n) + \frac{24n-1}{12}p(n).
\end{align}
Now we can describe the action of the Hecke operator $T(25)$ as follows:
\begin{align*}
    M^+(\tau)\mid T(25)=&- \frac{5}{12}q^{-25}+
    \frac{1}{12}q^{-1}
    + \sum_{n\geq 1} 5 m(24n-1) q^{25(24n-1)}\\
    &+ \sum_{n\geq 1} \left[m(25(24n-1)) - \left( \frac{-24n+1}{5} \right) m(24n-1) \right] q^{24n-1}.
\end{align*}
Since $M_5(\tau) = M^+(\tau) \mid T(25) + 6M^+(\tau)$, we have
%To solve for $M_5(\tau)$, we need only add $6M(\tau)$ to our expression. We have
\begin{align*}
    M_5(\tau) = &-\frac{5}{12}q^{-25} - \frac{5}{12}q^{-1} + \sum_{n\geq 1} 5 \left[m(24n-1) \right] q^{25(24n-1)} \\&+ \sum_{n\geq 1} \left[m(25(24n-1))+\left[ 6-\left(\frac{-24n+1}{5} \right) \right] m(24n-1) \right] q^{24n-1}.
\end{align*}
Thus, when $\ell = 5$, the statement of Theorem~1.1 reduces to 
\begin{align*}
    M_5(\tau) = -\frac{5}{12} \eta^{-1}(24\tau) \left[q^{-24} - \sum_{n \geq 1} n c(n) q^{24n} \right]
\end{align*}
and we are able to rearrange as follows:
\begin{align*}
    -q^{-24} + \sum_{n \geq 1} n  c(n) q^{24n} &= \eta(24\tau) \Big[ -q^{-25} - q^{-1} \\
    &\hspace{0.4cm} + \frac{12}{5}\sum_{n\geq 1} \Big[  m(25(24n-1)) + \Big[ 6- \Big(\frac{1-24n}{5} \Big) \Big] m(24n-1) \Big] q^{24n-1} \\
    &\hspace{0.4cm} + \sum_{n\geq 1}  12 m(24n-1) q^{25(24n-1)} \Big].
\end{align*}
Recall the definitions of $h_1(m)$ and $h_2(m)$ in (\ref{h coefficients}). Using these, we define
\begin{align*}
    \delta_1(n) &:= \sum_{k \in \Z}(-1)^k h_1(n-(6k+1)^2),
    \\
    \delta_2(n) &:= \sum_{k \in \Z}(-1)^k h_2(n-(6k+1)^2).
\end{align*}
Then we may write
\begin{align*}
    -q^{-24}+\sum_{n \geq 1} n c(n)q^{24n} &= \sum_{n \geq 1} \delta_1(24n) q^{24n} + \sum_{n \geq 1} \delta_2(24(25n-1))q^{24(25n-1)} \\
    &-\sum_{n=-\infty}^{\infty} (-1)^n q^{(6n+1)^2-25}-\sum_{n=-\infty}^{\infty}(-1)^nq^{(6n+1)^2-1}.
\end{align*}
We note that for $n \geq 1$,
\begin{align*}
    \sum_{n \geq 1} s(n)q^{24n}=-\sum_{n=-\infty}^{\infty} (-1)^n q^{(6n+1)^2-25}-\sum_{n=-\infty}^{\infty}(-1)^nq^{(6n+1)^2-1}.
\end{align*}
Thus, Theorem~\ref{thm: 1.2} follows by solving for $c(n)$.\hspace*{\fill}$\square$

\quad

\noindent \textit{Proof of Corollary \ref{cor: 1.5 c(n)}.}
This result follows immediately from Theorem~\ref{thm: 1.2} and the relation $spt(n) = -u^*(n) + 2a(n)$. 

\begin{rem}
While the results above use only the action of the specific Hecke operator $T(25)$, one should note that the entire sequence of operators $T(\ell^2)$ generate similar results for the polynomials $B_{\delta_\ell}(j(24\tau)$. We outline this process below. We define
\begin{align*}
    q\frac{d}{dq}j(24\tau)\cdot B_{\delta_\ell}(j(24\tau)):=\sum_{n\gg - \infty} r_\ell(n)q^{24n}.
\end{align*}
Then likewise if 
\begin{align*}
    m_\ell(24n-1)&:= m(\ell^2(24n-1))+\left( \frac{3}{\ell} \right) \left[\left( \frac{-(24n-1)}{\ell}\right)- (1+\ell) \right]m(24n-1) \\
    &\hspace{0.55cm}+\ell m((24n-1)/\ell^2),
\end{align*}
we may write
\begin{align*}
    M_\ell(\tau) &= -\frac{\ell}{12}q^{-\ell^2} +\left(\frac{3}{\ell} \right)\frac{\ell}{12} q^{-1} +\sum_{n \geq 1} m_\ell(24n-1) q^{24n-1}.
\end{align*}
Rewriting the result of Theorem 1.1, we have
\begin{align}
    \sum_{n\gg - \infty} r_\ell(n)q^{24n}=\frac{12}{\ell}\eta(24\tau)M_\ell(\tau).
\end{align}
Thus, expanding the right-hand side using the pentagonal number theorem allows one to solve for $r_\ell(n)$. By Theorem \ref{thm: 1.2} and Corollary \ref{cor: 1.5 c(n)}, the coefficients of $q\frac{d}{dq}j(24\tau)$ are known in terms of combinatorial quantities, and so the coefficients of $B_{\delta_\ell}(j(\tau))$ themselves can written as a sequence of combinatorial expressions as well.
\end{rem}

%Then scaling and multiplying by $\eta(24\tau)$ allows one to solve for the $r(n)$ in general.
%Lastly, we note that Corollary 1.5 immediately follows from the definition of $U(-1;q)$ and the preceding results in terms of $p(n)$ and $spt(n)$.

\section*{Acknowledgments}

We thank Professor Ken Ono for his invaluable mentorship and guidance in producing this paper. We also thank the reviewer for their careful reading and helpful suggestions. We additionally give thanks to the support of the Asa Griggs Candler Fund and the grants from the National Security Agency (H98230-19-1-0013), the National Science Foundation (1557960, 1849959), and the The Spirit of Ramanujan Talent Initiative.

\end{document}